\documentclass[10pt]{amsproc}
\usepackage{amsmath, amsthm,amssymb}

\usepackage[pagebackref, colorlinks, linkcolor=red, citecolor=blue, urlcolor=blue, hypertexnames=true]{hyperref}

\usepackage{setspace, amsrefs}

\allowdisplaybreaks

\title{A remark about the spectral radius}

\author{Andreas Thom}
\address{A.T., Mathematisches Institut, U Leipzig,
PF 100920, 04009 Leipzig, Germany}
\email{andreas.thom@math.uni-leipzig.de}

\date{\today}

\newtheorem{thm}{Theorem}

\newtheorem{cor}[thm]{Corollary}
\newtheorem{lem}[thm]{Lemma}
\theoremstyle{definition}

\newtheorem{rem}[thm]{Remark}

\newcommand{\C}{{\mathbb C}}

\newcommand{\N}{{\mathbb N}}

\newcommand{\R}{{\mathbb R}}

\begin{document}

\onehalfspace

\maketitle
\begin{abstract} We show that for any finitely generated non-amenable group and any $\varepsilon>0$, there exists some finite symmetric generating set with spectral radius less than $\varepsilon$. We give applications to percolation theory and the theory of operator spaces.
\end{abstract}

\section{The spectral radius}

Let $\Gamma$ be a finitely generated group and let $S \subset \Gamma$ be a finite symmetric generating set. Here, a set is called symmetric if $S^{-1}=S$, i.e.\ $g \in S$ implies $g^{-1} \in S$. Consider the Hilbert space $\ell^2 \Gamma$ with orthonormal basis $\{\delta_g \mid g \in \Gamma \}$ and the left-regular representation $\lambda \colon \Gamma \to {\mathcal U}(\ell^2 \Gamma)$, which is defined by the formula $\lambda(g)\delta_h :=\delta_{gh}$. We will also use $\lambda$ to denote the linear extension $\lambda \colon \C[\Gamma] \to B(\ell^2 \Gamma)$. Here $\C[\Gamma]$ denotes the complex group ring and $\lambda$ is a $*$-homomorphism with respect to the natural involution on $\C[\Gamma]$. There is a natural trace $\tau \colon \C[\Gamma] \to \C$, defined by $\tau(\sum_g a_g g)=a_e$.

For each $a = \sum_g a_g g\in \C[\Gamma]$, we denote by $\|a\|$ the operator norm of the operator $\lambda(a) \in B(\ell^2 \Gamma)$. We also set
${\rm supp}(a) := \{g \in \Gamma \mid a_g \neq 0\}$, ${\rm size}(a) := |{\rm supp}(a)|$ and $\|a\|_1 := \sum_g |a_g|$. It is a basic property of the operator norm that $\|a^k\|=\|a\|^k$ whenever $a$ is a hermitean element, i.e.\ $a^*=a$. For $a,b \in \R[\Gamma]$, we write $a \leq_{\Gamma} b$ if $a_g \leq b_g$ for all $g \in \Gamma$. If $0 \leq_{\Gamma}b \leq_{\Gamma} a$ and $a,b \in \R[\Gamma]$ are hermitean, then $\|b\| \leq \|a\|$. Indeed, this follows from the spectral radius formula
\begin{equation} \label{eqsp}
\|b\| = \lim_{n \to \infty} \tau(b^{2n})^{1/2n} \leq \lim_{n \to \infty}\tau(a^{2n})^{1/2n} = \|a\|.
\end{equation}

Clearly, if $0 \leq_\Gamma a$, then $\|a^k\|_1 = \|a\|_1^k$.
For any symmetric subset $S \subset \Gamma$, we define the Markov operator $m(S) := \frac{1}{|S|} \sum_{s \in S} s$.
Kesten \cite{MR0112053} showed that the group $\Gamma$ is non-amenable if and only $\|m(S)\|< 1$ for some (and hence any) finite symmetric generating set of $\Gamma$. We will also use the notation $\rho(S) := \|m(S)\|$ and call it the spectral radius of the random walk associated with $S$. The spectral radius formula
$$\rho(S) = \lim_{n \to \infty} \tau(m(S)^{2n})^{1/2n}$$
gives the explanation for this terminology, since the right side of this equation is the exponential growth rate of the return probability of a random walk in the Cayley graph of $\Gamma$ with respect to the generating set $S$ after $2n$ steps, starting at the neutral element, see \cite{woess} for definitions and references.

Kesten's result gives a powerful criterion for amenability as well as non-amenability and has been used in many circumstances. 
It has been asked over the years if for any finitely generated non-amenable group and any $\varepsilon>0$ a finite symmetric generating set $S$ can be found such that $\|m(S)\| \leq \varepsilon$. (To trace back the precise history of this question is very difficult since it is so natural. For example, it has been asked by Gilles Pisier in his study of versions of the von Neumann problem, see Section \ref{opsp}.) Even though desirable to study and natural to ask, this question has remained unanswered.
Anyhow, the pressure to answer this question has been low, since in essentially all proofs and applications replacing $m(S)$ by $m(S)^k$ (and thus obtaining operator norm $\|m(S)^k\| = \|m(S)\|^k$) has proved to be sufficient for the purposes of the argument. In more combinatorial situations this amounts to a replacement of $S$ by the multi-set $S^{[k]}$ of all words of length $k$ in letters from $S$. In this note, we will show that the original question has a positive answer. Partial results about the same problem have been obtained by Juschenko-Nagnibeda \cite{juna}.

\vspace{0.2cm}

  For any subset $S \subset \Gamma$, we denote by $S^k$ the set of all products of $k$ factors from the set $S$.

\section{The main result}

Our main result is:

\begin{thm} \label{mainthm}
Let $\Gamma$ be a finitely generated non-amenable group and $\Sigma$ be a finite symmetric generating set. For every $k \in \N$ there exists a symmetric set $S_k \subset \Sigma^k$ such that
$$\rho(S_k) \leq 4k \cdot \ln(|\Sigma|) \cdot \rho(\Sigma)^k.$$

Moreover, for every $\varepsilon>0$, there exists a finite symmetric generating set $S \subset \Gamma$ such that $\rho(S) < \varepsilon$.
\end{thm}

We need the following lemma to proceed.

\begin{lem} \label{main}
Let $f \colon [0,1] \to [0,1]$ be non-zero monotone decreasing function such that $\int_{0}^1 f(y)\ dy \leq \frac13$. Then, there exists some $x_0 \in [0,1]$, such that
$$x_0f(x_0) \geq \frac{\int_{0}^1 f(x)\ dx}{-4 \cdot \ln\left(\int_{0}^1 f(x)\ dx \right)}.$$
\end{lem}
\begin{proof}
We set $I:= \int_{0}^1 f(y)\ dy>0$. First of all, we define $\alpha := \frac1{-4 \ln(I)}>0$ and note that $I \leq \frac13$ implies $\alpha \leq \frac{1}{4 \ln(3)}\leq \frac14$. Assume that we have
$x f(x) < \alpha  I$ for all $x \in [0,1].$
Then, since $f$ is decreasing and $f(0) \leq 1$, we have
$$f(x) \leq \min \left \{1,\frac{\alpha I}x  \right \}, \quad \forall x \in [0,1]$$ and thus, we can estimate the integral as follows:
\begin{equation} \label{eq1}
I \leq \alpha  I + \alpha  I \cdot \int_{\alpha  I}^1 \frac1x\ dx = \alpha  I - \alpha  I \cdot \ln(\alpha  I).
\end{equation}
Using $\alpha \leq \frac14$, we get
$$\frac3{4\alpha} \leq \frac{1-\alpha}{\alpha} \stackrel{\eqref{eq1}}{\leq}  - \ln(\alpha  I).$$ 
Since $- \ln(\alpha) < \frac{1}{2\alpha}$ for all $\alpha \in (0,\infty)$, the inequality above implies
\begin{equation} \label{eq2}
\frac{1}{4\alpha}  < - \ln(I),
\end{equation}
in contradiction with our choice of $\alpha$. This proves the claim.
\end{proof}

\begin{rem}
The inequality in Lemma \ref{main} cannot be improved to $x_0 f(x_0) \geq \alpha \int_{0}^1 f(y)\ dy$ for a fixed $\alpha>0$ independent of $f$. This follows from a computation with the functions $f_n(x) = \min\{1,\frac{1}{nx}\}$ for $n \in \N$.
\end{rem}

We say that $b \in \R[\Gamma]$ is a one-step function if the set $\{b_g \mid g \in \Gamma, b_g\neq 0 \}$ is a singleton, i.e.\ $b$ is a multiple of a characteristic function. The main idea in the proof of Theorem \ref{mainthm} is to find a large one-step function with non-negative coefficients which is coefficient-wise smaller than a suitable power of the Markov operator.

\begin{cor} \label{cor}
Let $\Gamma$ be a group and let $a \in \R[\Gamma]$ be a hermitean element with non-negative coefficients and ${\rm size}(a) \geq 3.$ Then, there exists a hermitean one-step function $b \in \R[\Gamma]$ such that
$0 \leq_{\Gamma} b \leq_{\Gamma} a$ and
$$\|b\|_1 \geq \frac14 \cdot \frac{\|a\|_1}{\ln ( {\rm size}(a))}.$$
\end{cor}
\begin{proof} We write $a = \sum_g a_g g$ with $a_g \in [0,\infty)$. All coefficients of $a$ are bounded from above by $\|a\|_1$. Choose an enumeration $n \mapsto g_n$ of the set ${\rm supp}(a)$ such that $n \mapsto a_{g_n}$ is monotone decreasing and set
$f(x) := a_{g_n}\|a\|^{-1}_1$ if ${\rm size}(a) \cdot x$ lies in the interval  
$[n-1,n)$.
Thus, we have got a monotone decreasing function $f \colon [0,1] \to [0,1]$ and
$$\int_{0}^1 f(y)\ dy = \frac{1}{{\rm size}(a)} \leq \frac13.$$
By Lemma \ref{main}, there exists $x_0 \in [0,1]$ such that
$$x_0f(x_0) \geq \frac{\int_{0}^1 f(x)\ dx}{-4 \cdot \ln\left(\int_{0}^1 f(x)\ dx \right)}.$$
Now, the function 
$$g(x):= \begin{cases} f(x_0) & x \in [0,x_0] \\
0 & x \in (x_0,1] \end{cases}$$
satisfies $0 \leq g(x) \leq f(x)$ for all $x \in [0,1]$ and $\int_{0}^1 g(y) \ dy= x_0 f(x_0)$.
Thus, by renormalization we find a one-step function $b \in \R[\Gamma]$ with $0 \leq_{\Gamma }b \leq_{\Gamma} a$ and
$$\|b\|_1 \geq \frac{\|a\|_1}{4 \cdot \ln ({\rm size}(a))}.$$
Since the level-sets of $a$ are symmetric subsets of $\Gamma$, we may choose $b$ hermitean. This finishes the proof.
\end{proof}

We are now ready to prove our main result.

\begin{proof}[Proof of Theorem \ref{mainthm}:] Consider the Markov operator $m(\Sigma) \in \R[\Gamma]$. Since $\Gamma$ is not amenable, we have ${\rm size}(m(\Sigma)) = |\Sigma| \geq 3$. Let $b_k \in \R[\Gamma]$ be the non-negative hermitean one-step function, whose existence is predicted by Corollary \ref{cor} applied to $m(\Sigma)^k \in \R[\Gamma]$. We set $S_k := {\rm supp}(b_k) \subset \Sigma^k$. Then
$$\rho(S_k) = \frac{\|b_k\|}{\|b_k\|_1} \leq \frac{4 \cdot \ln(|\Sigma^k|)}{\|m(\Sigma)^k\|_1}  \cdot\|b_k\|\stackrel{\eqref{eqsp}}{\leq} 4 \cdot \ln(|\Sigma^k|) \cdot \|m(\Sigma)^k\| \leq 4k \cdot \ln(|\Sigma|) \cdot \rho(\Sigma)^k.$$
This proves the first claim.

It remains to show that we also find finite symmetric generating sets, whose associated spectral radius is as small as we want.  Take $S'_n := S_n \cup \Sigma$, then
$$\rho(S'_n) \leq \frac{|S_n|}{|S_n \cup \Sigma|} \cdot \rho(S_n) + \frac{|\Sigma \setminus S_n|}{|S_n \cup \Sigma|} \cdot \rho(\Sigma)$$
by the triangle inequality. For all finite symmetric subsets $\rho(S) \geq |S|^{-1/2}$ by comparison with a regular tree, so that $|S_n|^{-1} \leq \rho(S_n)^2$ for all $n \in \N$. We get
$$\rho(S'_n) \leq \rho(S_n) + |\Sigma| \cdot \rho(\Sigma) \cdot \rho(S_n)^2 = O(k \rho(\Sigma)^k).$$ This finishes the proof.
\end{proof}

\begin{rem}
If $\Gamma$ contains a free subgroup, then the first part of the proof of Theorem \ref{mainthm} is trivial. Indeed, any free group contains free groups of arbitrary rank. Moreover, if $S_n = \{a_1,a_1^{-1},\dots,a_n,a_n^{-1}\}$, where $a_1,\dots,a_n$ are the basis for a free group, then it is well-known that $\rho(S_n) = O(n^{-1/2})$.
\end{rem}

\section{Some applications}

\subsection{Quantitative statements}

As a consequence of Theorem \ref{mainthm}, we can prove the following characterization of amenability.
\begin{cor} \label{char}
Let $\Gamma$ be a group. Suppose that for all $\varepsilon>0$, there exists $n(\varepsilon) \in \N$ such that for all finite symmetric subset $S \subset \Gamma$ with $|S| \geq n(\varepsilon)$, we have $\rho(S) \geq |S|^{-\varepsilon}$. Then, $\Gamma$ is amenable. 

Equivalently, for any non-amenable group, there exists some $\varepsilon>0$, such that there exist arbitrarily large finite symmetric subsets $S$ with $\rho(S) < |S|^{-\varepsilon}$.
\end{cor}
\begin{proof}Without loss of generality, we may assume that $\Gamma$ is finitely generated.
We prove the second version.
Indeed, suppose that $\Sigma$ is a finite symmetric generating set and $\Gamma$ is non-amenable. Then $\rho(\Sigma)<1$ and we set $\varepsilon:= - \frac{\ln \rho(\Sigma)}{2\ln |\Sigma|}>0$. Let $S_k$ be a sequence of finite subsets as in the statement of Theorem \ref{mainthm}. Then, we get with $n:= |\Sigma|^k$:
$$\rho(S_k) \leq 4 \ln(n) \cdot \rho(\Sigma)^k =4 \ln(n) \cdot n^{\frac{\ln(\rho(\Sigma))}{\ln(|\Sigma|)}} = \frac{4\ln(n)}{n^{2\varepsilon}} < n^{-\varepsilon} \leq |S_k|^{-\varepsilon},$$
where the second inequality holds if $n$ (and hence $k$) is large enough. This proves the claim.
\end{proof}

As a direct consequence we get:

\begin{cor} Let $f \colon \N \to [0,\infty)$ be a monotone increasing function with $f(n)=o(\ln(n))$. 
Let $\Gamma$ be a group. If $\rho(S) \geq \exp(-f(|S|))$ for all finite symmetric subsets $S \subset \Gamma$, then $\Gamma$ is amenable.
In particular, if $\rho(S) \geq \frac{1}{\ln(|S|)}$ for all finite symmetric subsets $S \subset \Gamma$, then $\Gamma$ is amenable.
\end{cor}
\begin{proof}
Indeed, if $f(n) = o(\ln(n))$, then $\exp(-f(n)) \geq n^{-\varepsilon}$ eventually for all $\varepsilon>0$. Hence, the result follows from Corollary \ref{char}. The second claim follows taking $f(n)=\ln(\ln(n))$.
\end{proof}

\subsection{Percolation theory} We can improve on \cite[Theorem 1]{MR1756965} and can get around the implicit use of multi-sets in the statement of the theorem. This settles a particular case of a conjecture by Benjamini-Schramm, see \cite{MR1423907, peres} for definitions and background on percolation theory.
	
\begin{cor}
Let $\Gamma$ be a finitely generated non-amenable group. Then there exists a finite symmetric set of generators $S$
in $\Gamma$ such that $p_c(\Gamma,S) < p_u(\Gamma,S)$.
\end{cor}
\begin{proof} Having Theorem \ref{mainthm} at hand, the proof proceeds as in \cite{MR1756965}.
\end{proof}

In a similar direction we obtain:
\begin{cor}
Let $\Gamma$ be a finitely generated non-amenable group. Then the expected degree of ${\rm FMSF}(\Gamma,S)$ is unbounded when $S$ varies among finite symmetric generating sets. In particular, the expected degree of ${\rm FMSF}(\Gamma,S)$ depends on $S$ and there exists a finite symmetric generating set $S$ such that ${\rm WMSF}(\Gamma,S) \neq {\rm FMSF}(\Gamma,S)$. 
\end{cor}
\begin{proof} This is a consequence of Theorem \ref{mainthm} and \cite[Corollary 2]{thomdegree}.
\end{proof}

\subsection{Operator spaces}
\label{opsp}
For background on operator spaces consult Pisier's book \cite{pisier}. Let ${\mathbb F}_n$ be the free group on generators $g_1,\dots,g_n$ and $C_{\rm red}^*({\mathbb F}_n)$ denote its reduced group $C^*$-algebra. We denote by $E_n$ the operator space spanned by $\lambda(g_1),\dots,\lambda(g_n) \in C_{\rm red}^*({\mathbb F}_n)$. For any operator space $F$, we set 
$$\gamma_{F}(E_n) := \inf \|u\|_{\rm cb}\|v\|_{\rm cb},$$
where the infimum is taken over all factorizations $E_n \stackrel{u}{\to} F \stackrel{v}{\to} E_n$ of the identity ${\rm id_n} \colon E_n \to E_n$. 

Let $\Gamma$ be a countable group and let $L(\Gamma)$ be its group von Neumann algebra. Haagerup-Pisier \cite{haapis} proved that if $\Gamma$ is amenable group, then
$\gamma_{L(\Gamma)}(E_n) \geq \frac{n^{1/2}}2$ for all $n \in \N$. Pisier \cite{pisiernotes} raised the question if there is a converse to this result and in this context he also asked if Theorem \ref{mainthm} could be proved.
\begin{thm} \label{appop}
Let $\Gamma$ be a countable non-amenable group. There exists $\delta>0$, such that
$$\liminf_{n \to \infty} \frac{\gamma_{L(\Gamma)}(E_n)}{n^{1/2-\delta}} = 0.$$
\end{thm}
\begin{proof} 
Without loss of generality, we may assume that $\Gamma$ is finitely generated. The following reduction to Theorem \ref{mainthm} is due to Pisier \cite{pisiernotes}.
Let $S = \{t_1,\dots,t_n\} \subset \Gamma$ with $|S|=n$.
We define $u \colon E_n \to L(\Gamma)$ by $\lambda(g_i) \mapsto \lambda(t_i)$. We claim that
$\|u\|_{cb} \leq \left\|\sum_{i=1}^n t_i \right\|^{1/2}.$ Indeed,
\begin{eqnarray*}
\left\|\sum \lambda_{\Gamma}(t_i) \otimes a_i\right\|_{B(\ell^2\Gamma \otimes H)}^2 & \leq & \left\|\sum \lambda_{\Gamma}(t_i) \otimes \lambda_{\Gamma}(t_i) \right\|_{B(\ell^2\Gamma \otimes \ell^2\Gamma)} \left\|\sum \bar a_i \otimes  a_i \right\|_{B(\bar H \otimes  H)}  \\
& = & \left\|\sum t_i \right\| \cdot  \left\|\sum \bar a_i \otimes  a_i \right\|_{B(\bar H \otimes  H)} \\
&\leq & \left\|\sum t_i \right\| \cdot \max \left\{ \left\|\sum a_i^*a_i\right\|_{B(H)}, \left\|\sum a_ia_i^* \right\|_{B(H)} \right\} \\
& \leq & \left\|\sum t_i \right\| \cdot  \left\|\sum \lambda_{\mathbb F_n}(g_i) \otimes a_i \right\|_{B(H \otimes \ell^2 {\mathbb F_2})}^2,
\end{eqnarray*}
where the first line is based on Haagerup's variant of the Cauchy-Schwarz inequality \cite[Lemma 2.4]{haagerup} (see for example \cite[p.\ 123]{pisier}), the second is obvious, the third is \cite[Equation (2.11.4)]{pisier}, and the fourth was proved in \cite[Proposition 1.1]{haapis}. This proves the claim.

We define $v \colon L(\Gamma) \to E_n$ on a dense subspace by $v(\lambda_{\Gamma}(t_i))= \lambda_{\mathbb F_n}(g_i)$ for $1 \leq i \leq n$ and $v(\lambda_{\Gamma}(t))=0$ if $t \not \in  \{t_1,\dots,t_n\}$.
It is easy to see that for any finitely supported function $f \colon \Gamma \to B(H)$, we have
$$\left\| \sum f(t)^*f(t)\right\|_{B(H)} \leq \left\| \sum  \lambda_{\Gamma}(t) \otimes f(t)\right\|^2_{B(\ell^2\Gamma \otimes H)}$$ and similarly
$$\left\| \sum f(t)f(t)^*\right\|_{B(H)} \leq \left\| \sum   \lambda_{\Gamma}(t) \otimes f(t)\right\|^2_{B(\ell^2\Gamma \otimes H)}.$$
Thus, 
\begin{eqnarray*}
\left\|\sum \lambda_{\mathbb F_2}(g_i) \otimes f(t_i) \right\|_{B(\ell^2{\mathbb F_2} \otimes H)} &\leq & 
2 \cdot \max \left\{ \left\|\sum f(t_i)^*f(t_i)\right\|^{1/2}_{B(H)}, \left\|\sum f(t_i)f(t_i)^* \right\|^{1/2}_{B(H)} \right\}\\
&\leq&
2 \cdot \left\|\sum \lambda_{\Gamma}(t) \otimes f(t)\right\|_{B(\ell^2\Gamma \otimes H)},
\end{eqnarray*}
where we have used \cite[Proposition 1.1]{haapis} in the first line. Hence, $v$ is indeed well-defined on $L(\Gamma)$ and $\|v\|_{cb} \leq 2$.
We can conclude that $\gamma_{L(\Gamma)}(E_n) \leq 2 \left\| \sum t_i \right\|^{1/2}$. 

Now, there exists  $\varepsilon>0$ as in Corollary \ref{char} and a sequence of finite symmetric subsets $S_k \subset \Gamma$, such that $\rho(S_k)  \leq |S_k|^{-\varepsilon}$ as $k \to \infty$. Then, we can set $n_k := |S_k|$, $\delta< \varepsilon/2$ and notice that $$0 \leq \limsup_{k \to \infty} \frac{\gamma_{L(\Gamma)}(E_{n_k})}{n_k^{1/2-\delta}} \leq \lim_{k \to \infty} 2\rho(S_k)^{1/2} |S_k|^{\delta} =0.$$ This finishes the proof.
\end{proof}

There are various other questions (some of them outstanding) in this context -- in particular concerning characterizations of amenability -- and we hope that Theorem \ref{mainthm} will have more applications.

\subsection{A variation of the technical lemma}

An immediate and potentially useful consequence of Lemma \ref{main} is the following result:

\begin{cor} Let $(X,\mu)$ be a probability measure space and $f \colon X \to [0,1]$ be a non-zero measurable function such that $\|f\|_1 \leq \frac13$. Then, there exists $\lambda \in [0,1]$, such that

$$\lambda \cdot \mu (\{ x \in X \mid f(x) \leq \lambda \}) \geq \frac{\|f\|_1}{ -4 \cdot \ln(\|f\|_1)}.$$
\end{cor}
\begin{proof}
Indeed, there exists a measure preserving surjection $\varphi \colon ([0,1],\mu_{{\rm Leb}}) \to (X,\mu)$ and a measure preserving automorphism $\alpha$ of $([0,1],\mu_{{\rm Leb}})$ such that $[0,1] \in x \mapsto f(\varphi(\alpha(x)))$ is monotone decreasing. This reduces the statement to the case studied in Lemma \ref{main}. 
\end{proof}

\section*{Acknowledgments}

I am grateful to Gilles Pisier for explanations, sharing his manuscript \cite{pisiernotes}, and letting me include his argument in the proof of Theorem \ref{appop}.
I want to thank Nicolas Monod and Kate Juschenko for interesting remarks.

\end{document}